\documentclass[11pt]{amsart}
\usepackage{pstricks} 
\usepackage{color}
\usepackage[pagewise]{lineno}
\usepackage{accents}
\usepackage{amsthm}
\usepackage[all]{xy}
\usepackage{amssymb,latexsym}
\usepackage{graphicx}

\usepackage{multirow}

\renewcommand{\phi}{\varphi}

\newcommand{\R}{\ensuremath{\mathbb{R}}}
\newcommand{\C}{\ensuremath{\mathbb{C}}}
\newcommand{\zed}{\ensuremath{\mathbb{Z}}}
\newcommand{\A}{\ensuremath{\mathcal{A}}}

\newcommand{\pstar}{\ensuremath{P_*}}

\newcommand{\reals}{\mathbb{R}}

\newcommand{\im}{\operatorname{im}}

\newcommand{\FPhat}{\ensuremath{\hat{F}_P}}
\newcommand{\FQhat}{\ensuremath{\hat{F}_Q}}
\newcommand{\X}{X}
\newcommand{\Z}{Z}
\newcommand{\M}{M}

\newcommand{\NCP}{\ensuremath{\operatorname{NCP}}}
\newtheorem{theorem}{Theorem}[section]
\newtheorem{proposition}[theorem]{Proposition}
\newtheorem{corollary}[theorem]{Corollary}
\newtheorem{lemma}[theorem]{Lemma}
\theoremstyle{definition}
\newtheorem{ack}{Acknowledgements}

\newtheorem{definition}[theorem]{Definition}
\newtheorem{remark}[theorem]{Remark}
\newtheorem{example}[theorem]{Example}

\numberwithin{equation}{section}
\begin{document}
\title[Non-crossing partitions and Milnor fibers]
{Non-crossing partitions and Milnor fibers}

\author[Brady]{Thomas~Brady}
\address{School of Mathematical Sciences\\
Dublin City University\\
Glasnevin, Dublin 9\\
Ireland}
\email{tom.brady@dcu.ie}

\author[Falk]{Michael~Falk$\>^{*}$}
\address{Department of Mathematics and Statistics\\
Northern Arizona University\\
Flagstaff,  AZ  86011 \\
U.S.A.}
\email{michael.falk@nau.edu}
\thanks{${\>\kern-\parindent}^*\,$supported in part by a Fulbright U.S. Scholars grant}

\author[Watt]{Colum~Watt}
\address{School of Mathematical Sciences\\
Dublin Institute of Technology\\
Dublin 8\\
Ireland}
\email{colum.watt@dit.ie}

\date{\today}
\subjclass{Primary 20F55;\, Secondary 52C35}
\thanks{ 2000 \textit{Mathematics Subject Classification.} Primary
20F55; \, Secondary 05E99.}
\keywords{Milnor fibers, finite reflection groups, generalised braid groups, non-crossing partitions}

\begin{abstract}
For a finite real reflection group $W$ we use non-crossing partitions of type $W$ to construct finite cell complexes 
with the homotopy type of the Milnor fiber of the associated $W$-discriminant $\Delta_W$ and that of the Milnor fiber of  the defining polynomial of the associated reflection 
arrangement. These complexes support natural cyclic group actions realizing the geometric monodromy. 
Using the shellability of the non-crossing partition lattice, this cell complex yields a chain complex of 
homology groups computing the integral homology of the Milnor fiber of $\Delta_W$.
\end{abstract}

\maketitle

\section{Introduction}
\label{intro}
Suppose $g\in \C[z_1, \ldots, z_n]$ is a quasi-homogeneous polynomial, 
defining the hypersurface $V=g^{-1}(0)$ in $\C^n$. Then $g$ restricts to 
a locally trivial fibration $g \colon \C^n- V \to \C^*$, the global Milnor fibration, 
with fiber $g^{-1}(1)$, the Milnor fiber of $g$ \cite{Mi68}. The topology of $g^{-1}(1)$
and the monodromy of the bundle are invariants of the singularity type of $g$ at the origin. 
Of special interest is the case where $g=Q_W$ is a product of complex linear forms defining 
the arrangement $\A=\A_W$ of reflecting hyperplanes in $\C^n$ of a finite real or complex 
reflection group $W$, see \cite{Sett04, Sett09, MacPap09,DimLeh16}.


In this setting $W$ acts on $\C^n$ preserving $V=\bigcup_{H\in \A} H$, 
the quotient $W\backslash \C^n$ is homeomorphic to $\C^n$, and under this homeomorphism $W \backslash V$ is 
carried to a hypersurface $\Delta_W$ in $\C^n$.  
This hypersurface is the zero 
locus of a quasi-homogeneous polynomial, $P_W$, well-defined up to polynomial automorphism of $\C^n$, called the discriminant associated with $W$ (see Section~\ref{discriminant}). 

The fundamental group of $\C^n - \Delta_W$ is 
the generalized braid group (or Artin group), $B(W)$, associated with $W$. 
If $W$ has type $A_{d-1}$ then $P_W$ is the classical discriminant for 
polynomials of degree $d$ and $B(W)$ is isomorphic to the classical
 braid group on $d$ strands.  In this paper, we construct a non-crossing partition (\NCP) model for
 the Milnor fiber $F_P=P_W^{-1}(1)$ and study its structure, including the monodromy action, 
 in the case where $W$ is a real reflection group.  
 We also construct an 
 \NCP\ model for the Milnor fiber $F_Q=Q_W^{-1}(1)$ of the reflection arrangement $\A_W$. 
Both models  arise as subcomplexes of appropriate covering spaces of a 
finite $K(B(W), 1)$ which is  defined in terms of non-crossing partitions, see \cite{ Be, B, BW}. 

The \NCP\ model for $F_P$ has a natural filtration by subcomplexes, which are seen to be homotopy equivalent to bouquets of spheres using the lexicographic shellability of the non-crossing partition lattice. This yields a chain complex computing $H_*(F_P,\Z)$ whose terms are homology groups of truncations of this lattice. The homology of $H_*(F_P,\Z)$ has been computed in most cases, see \cite{CalSal04, Ca06}.
\section{\NCP\ models for subgroups of $B(W)$}
\subsection{Background} 
\label{background} Let $W$ be  a finite, irreducible, real reflection group of rank $n$.  
For background on finite reflection groups see \cite{Bou, Hum90}.   
Equip $W$ with the 
total reflection length function $w \mapsto |w|$ and with the partial order $\le$ given by 
$u \le w$ whenever $|u|+|u^{-1}w| = |w|$ (see \cite{DA}).  
We will use the notation $u\lessdot w$ for the case where $w$ covers $u$.
Fix a specific 
Coxeter element $\gamma$ in $W$  and define the non-crossing partitions 
to be the elements in the interval $[e,\gamma]$ in the poset $(W,\le)$.  
The poset of $W$-non-crossing partitions is a lattice, $L$ (see \cite{BrWa08}), 
whose order complex is denoted  $|L|$.  When 
$W$ is of type $A_n$, $W$ is isomorphic to the group of permutations of
$\{1, \ldots, n+1\}$ and the non-crossing partitions are those
elements whose cycle structure gives a 
classical non-crossing partition, see \cite{B}.

We define 
$B(W)$ to be the group with generating set   
 \[\{[w]\mid w \in L , w\ne e\}\]
subject to the relations
\[[w_1][w_1^{-1}w_2] = [w_2] \  \mbox{whenever} \ w_1\le w_2.\]
It is shown in \cite{Be, B, BW}, that $B(W)$ is isomorphic to the generalized braid group of type $W$.
\vskip .2cm  

We recall from \cite{Be, B, BW}, the contractible, $n$-dimensional, simplicial complex, $\X$,
whose $k$-simplices are ordered $(k+1)$-tuples from $B(W)$ of the form  
$(g_0, g_1, \dots , g_k)$ with $g_i = g_0[w_i]$ for some chain $e<w_1<w_2< \dots < w_k$ in $L$.  
It is convenient to use the notation $(g_0, e< w_1< \dots <w_k)$ for such a simplex.   
Thus the simplices of $\X$ are identified with 
pairs $(g,\sigma)$ where $g \in B(W)$ and $\sigma$ is an \textbf{initialized chain} in $L$,
that is,
$\sigma$ is a chain of the form $e<w_1< \dots < w_k$. As $B(W)$ acts freely on $X$, the quotient 
$K:= B(W)\backslash X$ is a $K(B(W),1)$ and $\X$ is its universal
cover. The action of $B(W)$ on $\X$ is given by 
\[g\cdot(g_0, g_1, \dots , g_k) = (gg_0, gg_1, \dots , gg_k),\]
or, in terms of the pair notation, 
\begin{equation}\label{action}
g\cdot(g_0, e< w_1< \dots <w_k) = (gg_0, e< w_1< \dots <w_k).
\end{equation}
It is immediate that the simplex $(g_0,e< w_1< \dots <w_k)$ has $k$ faces 
of the form 
\((g_0,e< w_1< \dots <\widehat{w_i}< \dots <w_k), \ \mbox{for}\ 1\le i \le k,\)
each obtained by deleting one of $w_1, w_2, \dots , w_k$ from $\sigma$.  
The remaining face is obtained by deleting $e$ from $\sigma$ and hence 
is given by the ordered set
\[(g_0[w_1], g_0[w_2], \dots , g_0[w_k]) = (g_0[w_1])\cdot(e, [w_1]^{-1}[w_2], \dots , [w_1]^{-1}[w_k]).\]
In pair notation, this is denoted $(g_0[w_1], e< w_1^{-1}w_2< \dots < w_1^{-1}w_k)$.
\subsection{Quotients of $\X$}
\label{quotients}
If $H$ is a normal subgroup of $B(W)$ then we can form a CW complex $\X_H$,  whose cells are of the form $(Hg,\sigma)$, where $\sigma$ is an initialized chain in $L$ 
and the first component is a right $H$ coset.  
If $\sigma =   e< w_1< \dots <w_k$, 
then this cell has $k$ boundary faces of the form 
\[(Hg,e< w_1< \dots <\widehat{w_i}< \dots <w_k), \ \mbox{for}\ 1\le i \le k,\]
with the remaining face given by $(Hg[w_1], e< w_1^{-1}w_2< \dots < w_1^{-1}w_k)$.  
It will be convenient to refer to $(Hg[w_1], e< w_1^{-1}w_2< \dots < w_1^{-1}w_k)$ as the \textbf{top} face of $(Hg,\sigma)$ and to $(Hg,e< w_1< \dots < \dots <w_{k-1})$ as the \textbf{bottom} face of $(Hg,\sigma)$.  
Since $\X$ is contractible, $\X_H$ is a $K(H,1)$.  The  action of the quotient group $H\backslash B(W)$ on 
$\X_H$ is given by $(Hg_1)(Hg_2, \sigma) = (Hg_1g_2, \sigma)$.
\vskip .2cm
We now highlight a particular feature of these complexes $\X_H$.
\vskip .2cm
\begin{lemma} 
\label{incidence}
In $\X_H$, each  $k$-cell   of the form
\[c_k = (Hg,e< w_1< w_2< \dots < w_k)  \ \ \ \mbox{with $|w_k| < n$} \]
is incident on precisely two $(k+1)$-cells of the form
\[c_{k+1} = (Hg', e< u_1< \dots <u_k< \gamma).\]
\end{lemma}

\vskip .2cm
\begin{proof}  Suppose that the cell  
\[c_k = (Hg,e< w_1< w_2< \dots < w_k),\]
with $|w_k| < n$, is incident on a $(k+1)$-cell of form
\[c_{k+1} = (Hg', e< u_1< \dots <u_k< \gamma).\]
Since the chain of $c_k$ does not contain $\gamma$, $c_k$ must be obtained by deleting either $e$ or $\gamma$ 
from the chain of $c_{k+1}$. 
In the latter case, $c_k$ is the bottom face of $c_{k+1}$, 
forcing 
$Hg'=Hg$
and $u_i = w_i$ for $i = 1, \dots , k$.  In the
former case, $c_k$ is the top face 
\[
(Hg'[u_1], e< u_1^{-1}u_2< \dots <u_1^{-1}u_k<u_1^{-1}\gamma),\]
so that 
$Hg'=Hg[w_1]^{-1}$ and
$u_1 = \gamma w_k^{-1}, u_2 = \gamma w_k^{-1}w_1, \dots , u_k = \gamma w_k^{-1}w_{k-1}$.
\end{proof}

In our examples, $H$ will arise as the kernel of a specific homomorphism $\phi$ 
with domain $B(W)$. In this case we will denote $X_{\ker(\phi)}$ by $X_\phi$.   
We can then identify the coset $Hg$ with the element $\phi (g)$ and denote the cells of $\X_H$ as 
pairs $(\phi(g), \sigma)$ where $\sigma$  is an initialized chain in $L$.  
If $\sigma =   e< w_1< \dots <w_k$, 
then the cell $(\phi(g), \sigma)$ has top face given by 
$(\phi(g[w_1]), e< w_1^{-1}w_2< \dots < w_1^{-1}w_k)$ or 
$(\phi(g)\phi([w_1]), e< w_1^{-1}w_2< \dots < w_1^{-1}w_k)$. 
\begin{example}
If $\phi$ is the trivial homomorphism, then $H = B(W)$ and $X_\phi = K$ is the 
$K(B(W), 1)$ introduced earlier.   
The cells of $K$ are of the form $(e, \sigma)$, where $\sigma$ is an initialized chain in $L$.   If 
$\sigma =   e< w_1< \dots <w_k$, 
then this cell has top face given by 
\[(e\phi([w_1]), e< w_1^{-1}w_2< \dots < w_1^{-1}w_k) = (e, e< w_1^{-1}w_2< \dots < w_1^{-1}w_k).\]  
Thus $\X=\X_\phi$ can be identified with the quotient of $|L|$ under the equivalence relation generated by identifying 
$w_1<w_2< \dots <w_k$ with $e<w_1^{-1}w_2< \dots < w_1^{-1}w_k$.
\end{example}
\begin{example}
\label{pure}The standard projection $s \colon B(W) \to W: [w] \mapsto w$, which
takes each \NCP\ generator of $B(W)$ to the corresponding \NCP\ in $W$, 
is a homomorphism by our presentation of $B(W)$. The kernel of $s$ is the pure braid 
group $PB(W)$ associated to $W$ and $X_s$ is a $K(\pi,1)$ for $\pi=PB(W)$.
(By \cite{Bries73,De72}, $X_s$ is homotopy equivalent to the complement 
$\M=\C^n - \bigcup_{H \in \A} H$, where \A\ is the complexification of the 
associated real reflection arrangement.)
The cells of $X_s$ can be identified with pairs $(w, \sigma)$, for 
$w \in W$ and $\sigma$ an initialized chain in $L$.  The  top face of the cell 
$(w, e<w_1<\dots <w_k)$ is $( ww_1, e<w_1^{-1}w_2<\dots <w_1^{-1}w_k)$. 
\end{example}

We will consider two further examples of this construction in sections \ref{disc} 
and \ref{fullfibre}.  These will give the models for the fibers mentioned in the introduction.  
The next section establishes the homotopy types of these fibers.

\section{Discriminants and Milnor fibers}
\label{discriminant}

In this section we recall some definitions and basic facts about discriminants and Milnor fibers. 

\subsection{The {\sl W}-discriminant} Recall that $W$ is a real reflection group whose
action on $\R^n$ has been complexified  to an action on 
$\C^n = \R^n \otimes \C$. Let $\{f_1, \ldots, f_n\}$ be a set of basic invariants 
for the ring of $W$-invariant, complex polynomials \cite{Che55}.  
The function $f=(f_1,\ldots,f_n) \colon \C^n \to \C^n$ induces
a homeomorphism of $W\backslash \C^n$ with $\C^n$.   
Let $T$ denote the set of reflections in $W$ and, for each  $t \in T$, let  $H_t \subset \C^n$
denote its fixed complex hyperplane and $\lambda_t \colon \C^n \to \C$ be a 
complex linear form 
with kernel $H_t$.
The polynomial $Q=\prod_{t\in T} \lambda_t \colon \C^n \to \C$ 
has the property that 
$Q(wx)=\det(w)Q(x)$ for all $w \in W, x \in \C^n$ and hence $Q^2$ is 
invariant under the action of $W$. It follows that  $Q^2=P(f_1, \ldots, f_n)$ for some 
quasi-homogeneous polynomial $P  \in \C[z_1, \ldots, z_n]$ whose weights are equal to the 
degrees of the $ f_i$. The polynomial $P$ is called the discriminant of $W$. It is unique up to polynomial automorphism of $\C^n$. 
The action of $W$ on $\C^n$ 
leaves the affine algebraic hypersurface 
$V=Q^{-1}(0)=\bigcup_{t \in T} H_t$ invariant and its quotient 
$\Delta_W = W\backslash V$ is identified with the affine algebraic 
hypersurface $\Delta := P^{-1}(0)$. 
The space $\M=\C^n - V$ is a $K(PB(W),1)$, and $W$ acts freely on $\M$, so the space 
$f(M)=W \backslash \M \cong \C^n - \Delta$ is a $K(B(W),1)$ \cite{Bries71,De72}. 
In what follows, we identify $PB(W)$ with $\pi_1(M,z_0)$ and $B(W)$ with $\pi_1(f(M),f(z_0))$.

\subsection{Milnor fibers of $P$ and $Q$} The restriction  to $\C^n - g^{-1}(0)$ of 
 any quasi-homogeneous polynomial 
$g \colon \C^n \to \C$ 
 is a locally trivial fibration  whose fiber $g^{-1}(1)$ 
 is called the Milnor fiber of $g$ \cite{Mi68}. The Milnor fibers of $P$ and $Q$ respectively
 will be denoted by $F_P$ and $F_Q$. 
 These spaces are determined up to polynomial diffeomorphism by $W$. 
Then $(Q^2)^{-1}(1)=Q^{-1}(1) \cup Q^{-1}(-1)$ is invariant  
under the action of $W$ and 
$F_P=W\backslash ((Q^2)^{-1}(1)) \cong W^+\backslash F_Q$, where 
$W^+=\{w \in W \mid \det(w)=1\}$. 
The space ${F_Q}$ is a connected, regular, $W^+$-cover of $F_P$
since the action of $W^+$ on $F_Q$ is free.

We will show that each of $F_P$ and $F_Q$ is homotopy equivalent to a complex of the
form $\X_\phi$. Our proofs will make use of  the following general result.

\begin{proposition}
\label{HomEquiv}
Suppose $g \colon E \to \C^*$ is a fibration with fiber $F$, where each of 
$E$ and $F$ has the homotopy type of a connected CW complex. Then $F$ is homotopy equivalent 
to the cover of $E$ corresponding to the kernel  of $g_* \colon \pi_1(E) \to \pi_1(\C^*)=\zed$.
\end{proposition}

\begin{proof} Let $i$ denote the inclusion of $F$ into $E$.
The exact sequence of the fibration implies that $i_* \colon \pi_1(F) \to \pi_1(E)$ 
is an injection whose image is $\ker(g_*)$. Let $p \colon F' \to E$ denote the connected cover of $E$ 
corresponding to $\ker(g_*)$. Then the inclusion $i$ lifts to a map $h \colon F \to F'$ with 
$p \circ h = i$. We show that $h$ is a homotopy equivalence. Since $p_*$ and $i_*$
are injections, it follows that 
 $h_* \colon \pi_1(F) \to \pi_1(F')$ is an injection and hence
 is an isomorphism. Since $p_* \colon \pi_k(F') \to \pi_k(E)$ 
and  $i_* \colon \pi_k(F) \to \pi_k(E)$ are  isomorphisms for all $k \geq 2$
(the latter by the exact sequence of the fibration), so also is 
 $h_* \colon \pi_k(F) \to \pi_k(F')$. As each of $F$ and $F'$ has the homotopy type of 
 a CW complex, $h$ is a homotopy equivalence.
\end{proof}

Since algebraic sets are homotopy equivalent to CW complexes, 
applying Proposition 3.1 yields the following corollary. 

\begin{corollary}\hskip -1mm \begin{enumerate}
\item The Milnor fiber $F_P$ of $P$ is homotopy equivalent to the cover of $f(M)=W\backslash M$ corresponding to
 the kernel of the map $P_* \colon B(W) \to \zed$.
\item The Milnor fiber $F_Q$ of $Q$ is homotopy equivalent to the cover of $f(M)=W\backslash M$ 
corresponding to the group~$f_*(\ker(Q_*))$.
\end{enumerate}
\label{fib}
\end{corollary}
\begin{proof} (ii): By proposition~\ref{HomEquiv}, $F_Q$ is homotopy equivalent 
to the connected domain,~$Y$, of a covering map $\rho \colon Y \to M$ for which
$\rho_*(\pi_1(Y,y_0)) = \ker(Q_*)$. Since $f \colon M \to W\backslash M$ is a finite cover,
the map $f \circ \rho$ is a covering map with $(f\circ \rho)_*(\pi_1(Y,y_0)) = f_*(\ker(Q_*))$,
as required.
\end{proof}

\subsection{The characteristic homomorphism} It remains to identify the 
homomorphisms  $Q_* \colon PB(W)=\pi_1(\M) \to \zed$ and 
$P_* \colon B(W)=\pi_1(W \backslash \M) \to \zed$. 
First we describe convenient generating sets for their respective domains. 
Fix $\epsilon > 0$. 
For each reflection $t \in T$, choose a point
$z_t \in H_t - \bigcup_{t'\in T, t' \neq t} H_{t'}$ and let $D_t$ be 
the closed disc of radius $\epsilon$ centred at $z_t$ in the 
complex line $L_t$ which passes through $z_t$ and  is
orthogonal to $H_t$.  The complex structure induces a natural orientation 
on each $L_t$. By shrinking $ \epsilon $, if necessary, we may assume that
$D_t \cap \bigcup_{t'\in T, t' \neq t} (H_{t'}\cup D_{t'}) = \emptyset $ for
all $t \in T$.  Now choose a basepoint $z_0$ in $\M$ and for each $t \in T$
choose a path $\sigma_t$ in $\M$ which starts at $z_0$ and ends on the boundary of 
$D_t$. Let $\gamma_t$ be the 
loop which travels along $\sigma_t$, then
around the boundary of $D_t$ in a positive orientation and finally back to 
$z_0$ along $\sigma_t$. The set of homotopy classes $\overline{\gamma}_t $ for 
$t \in T$ generates
$\pi_1(\M,z_0)$ \cite{Bries71}; see \cite{AllBas16} for a recent generalization. 

Since $W$ acts freely on $M$, the restriction of $f$ to $M$ is a regular covering map
onto $f(M)=W \backslash \M$. We specify a set of elements of  $\pi_1(W \backslash \M,f(z_0))$ 
which corresponds to
the set of reflections in $W \cong \pi_1(W \backslash \M,f(z_0)) / f_*\pi_1(\M,z_0)$
Under this isomorphism,  if $\delta$ is 
any path in $\M$ which starts at $z_0$ and 
ends at $w(z_0)$ (or which starts at $w(z_0)$ and ends at $w(w(z_0)$)),
then the homotopy class of $f \circ \delta$ corresponds to $w\in W$. 
If $t\in T$ is a reflection, let $\delta_t$ be the path in 
$\M$ which travels first along $\sigma_t$, then  half way around the 
boundary of $D_t$ in the positive sense and finally along the reverse of 
$t\sigma_t$ to $t(z_0)$. Similarly, let $\delta'_t$
be the path  from $t(z_0)$ to $z_0$ in 
$\M$ which travels first along $t \sigma_t$, then around the 
other half of the boundary of $D_t$ in the positive sense and finally along the reverse of 
$\sigma_t$ to $t(t(z_0))=z_0$. Note that $f \circ \delta_t$ and $f \circ \delta'_t$
represent the same homotopy class, $\Gamma_t$ say, and that the composition
$\Gamma_t$ followed by $\Gamma_t$ is represented by the path $f \circ \gamma_t$.
From the short exact sequence
\[\pi_1(\M,z_0) \to \pi_1(W \backslash \M,f(z_0))\to W\]
the group $\pi_1(W \backslash \M,f(z_0))$ is generated by
$\{ \overline{f \circ \gamma}_t,\Gamma_t: t \in T\}$ and, in fact,  
by the smaller set 
$\{\Gamma_t:t \in T\}$, since 
$\overline{f \circ \gamma}_t = \Gamma_t^2$. For $t \in T$, the generator
$\Gamma_t \in \pi_1(W \backslash \M,f(z_0))$ corresponds to the 
generator $[t]$ in the presentation for $B(W)$
from subsection~\ref{background}.
 
\begin{proposition} 
\hskip -1mm \begin{enumerate}
\item The map $Q_* \colon \pi_1(\M,z_0) \to \pi_1(\C^*,1) \cong \zed$ is given 
by\\ $Q_*(\overline{\gamma}_t)=1$ for each $t \in T$.
\item The map $P_* \colon \pi_1(W \backslash \M,f(z_0)) \to \pi_1(\C^*,1) \cong \zed$ 
is given 
by $P_*(\Gamma_t)=1$ for each $t \in T$.
\end{enumerate}
\label{uplink}
\end{proposition}

\begin{proof} (i) By scaling each $\lambda_t$ if necessary, we may assume that $\lambda_t(z_0) = 1$.
Then $\displaystyle Q_* = \left(\prod_{t \in T} \lambda_{t}\right)_{\!\!*} =
 \sum_{t \in T} (\lambda_{t*})$ (since $\C^*$ is a topological group).  The result now follows 
 since our choices ensure that
 the winding number of  $(\lambda_{t})_* ( \overline{\gamma}_{t'})$  about 
the origin is one if $t=t'$ and
zero otherwise.

\medskip
\noindent
(ii) First note that 
\[ P_*(\Gamma_t^2) = P_*(\overline{f \circ \gamma}_t) = (Q^2)_*(\overline{ \gamma}_t)
= Q_*(\overline{ \gamma}_t) + Q_*(\overline{ \gamma}_t)  = 2.\]
Since $P_*$ is a group homomorphism, it follows that $P_*(\Gamma_t) = 1$.
\end{proof}

\begin{remark} Our calculation of $P_*$ from $Q_*$ can be modified to apply when 
$W$ is a Shephard group, the symmetry group of a complex polytope. The construction 
of a model for the Milnor fiber of the $W$-discriminant (in the next section) 
will also be valid in this case.\end{remark}

\section{\NCP\ model for the Milnor fiber of the discriminant} 
\label{disc}
Consider the homomorphism $P_* \colon B(W) \to \zed$ and the cover of $K$ given by 
$\X_{\pstar}:= \ker(\pstar) \backslash \X$.
\begin{proposition}
\label{homotopytype}
$\X_{\pstar}$ is homotopy equivalent to the Milnor fiber of the discriminant $P$.
\end{proposition}
\begin{proof}  This follows from Corollary~\ref{fib}(i) and Proposition~\ref{uplink},
since corresponding covers of $K$ and $W\backslash M$ are homotopy equivalent.
\end{proof} 
\begin{remark}
By Section~\ref{quotients}, $\X_{\pstar}$ can be identified with the CW-complex 
whose cells are pairs, $(m, \sigma)$ for $m \in \zed$ 
and $\sigma$ an initialized chain in $L$.  
Since $P_*([t])=1$ for each reflection generator $[t]$ of $B(W)$,
it follows that $P_*([w])=|w|$, the reflection length of $w$, for each \NCP \ 
$w \in W$.
Hence, the top face of the cell $(m, e<w_1<\dots <w_k)$ is the cell 
\[(m+|w_1|, e<w_1^{-1}w_2<\dots <w_1^{-1}w_k).\]
\end{remark}
\begin{definition} We define the {\em small \NCP\ model} $\FPhat$ of the Milnor fiber of $P_W$ to be the finite subcomplex of $\X_{\pstar}$ consisting of the cells of the form 
\[(m, e<w_1<w_2<\dots < w_k) \ \text{with} \ 0\le m < n-|w_k|.
\]
\label{enn}
\end{definition}
\begin{remark} We observe that $\FPhat$ is the union of cells of the form
\[(0,e  \lessdot w_1 \lessdot w_2 \lessdot \dots \lessdot w_{n-1})\]
together with their faces. 
In particular, $\FPhat$ is $(n-1)$-dimensional.
\end{remark}
\begin{proposition}
\label{retraction}
The subcomplex $\FPhat$ is a strong deformation retract of $\X_{\pstar}$.
\end{proposition}
\begin{proof}  We construct an acyclic matching (see chapter 11 of \cite{K}) which pairs cells of 
$\X_{\pstar}- \FPhat$.  
Suppose $c_{k+1} = (m,\sigma)$ and that the chain $\sigma$ ends in $\gamma$. Then 
this matching pairs $c_{k+1} $ with 
its top (resp.~bottom) face if $m \ge 0$ (resp.~$m < 0$). In particular, the matching pairs 
cells whose chains end in $\gamma$ with cells whose chains do not end in $\gamma$.
\vskip .2cm 
To show that this matching is acyclic, consider an alternating path 
\[(l_1,\sigma_1) \succ_m (l_2,\sigma_2)\prec  (l_3,\sigma_3) \succ_m (l_4,\sigma_4)\prec \dots \]
in this matching, where 
$(l_i,\sigma_i)\prec  (l_{i+1},\sigma_{i+1})$ means that $(l_i,\sigma_i)$ is a facet of 
$(l_{i+1},\sigma_{i+1})$ and $(l_j,\sigma_j)\succ_m  (l_{j+1},\sigma_{j+1})$ means 
that $(l_j,\sigma_j)$ is matched with its facet $(l_{j+1},\sigma_{j+1})$.  
By definition of the matching, if $i$ is odd, the chain $\sigma_i$ ends in $\gamma$ 
and $\sigma_{i+1}$ 
does not end in $\gamma$.   By Lemma~\ref{incidence}, each  such $(l_{i+1},\sigma_{i+1})$ 
is incident on precisely two $(k+1)$-cells with chains ending in $\gamma$.  These 
must be  $(l_i,\sigma_i)$ and $(l_{i+2},\sigma_{i+2})$.  If $l_i<0$ it follows that  $(l_{i+1},\sigma_{i+1})$ is the bottom 
face of $(l_i,\sigma_i)$ and hence  $l_{i+2}> l_i$ for all $i$.  
Similarly, if $l_i \ge 0$ it follows that  $(l_{i+1},\sigma_{i+1})$ is the top face 
of $(l_{i+2},\sigma_{i+2})$ and hence  $l_{i+2}< l_i$ for all $i$.  
 In particular, the path 
cannot form a cycle and the matching 
is acyclic.  
\vskip .2cm
It remains to show that the set of critical cells is precisely 
the set of cells of $\FPhat$.    
Let $c_k = (m, e<w_1< \dots < w_k)$ be a cell of $\X_{\pstar}$.  
When $w_k =\gamma$, $c_k$ is matched with its top or bottom face 
according as $m \ge 0$ or $m <0$ and, hence, is not critical.  When $w_k \ne \gamma$,  
two cases arise.  If $m< 0$ then $c_k$ is matched as the bottom face of  
$(m, e<w_1< \dots < w_k< \gamma)$ and is not critical.  On the other hand,  if $m\ge 0$, 
then $c_k$ is matched as the top face of
\[(m-n+|w_k|,e< \gamma w_k^{-1}< \gamma w_k^{-1}w_1< \dots < \gamma w_k^{-1}w_{k-1}<\gamma)\] 
 if and only if $m-n+|w_k| \ge 0$.
\end{proof}

\begin{corollary} The finite complex $\FPhat$ is homotopy equivalent to the Milnor fiber $F_P$
and is a $K(\pi, 1)$ for $\pi=\ker(\pstar)$.
\end{corollary} 
\begin{example} \label{dihedral}
Let $W$ be a dihedral group acting on $\reals^2$ with $t$ reflections, 
$R_1, R_2, \dots , R_t$.  Thus the NCPs are $\{e, R_1, R_2, \dots , R_t, \gamma\}$, 
where $\gamma$ can be taken to be a rotation through twice the angle 
between adjacent lines of symmetry.  
Here $n = 2$ and $\FPhat$ is 
$1$-dimensional with $2$ vertices, namely $(0,e)$ and $(1,e)$.  $\FPhat$ has a $1$-cell $(0, e<R_i)$ for each chain $e<R_i$ and 
the endpoints of this  $1$-cell are $(0,e)$ and $(1,e)$.  Thus $\FPhat$ has the homotopy type of the suspension of a $0$-dimensional  
subcomplex on $t$ points and $\mbox{ker}(P_*)$ is free of rank $t-1$.  
This agrees with \cite[Theorem 1]{MilOr70}; $P$ has weights $2$ and $t$. 
\end{example}
\begin{example}  
Let $W$ be the group $A_3 \cong \Sigma_4$, the symmetric group; the polynomial $P$ is the classical discriminant for univariate quartics. Choose $\gamma$ to be the four-cycle $(1,2,3,4)$. 
Here $n = 3$ and $\FPhat$ is 
$2$-dimensional.  (See Figure \ref{triangle}.)  $\FPhat$ has $3$ vertices, namely $(0,e)$, $(1,e)$ and $(2,e)$.   
Each transposition $R$
contributes a $1$-cell $(0, e<R)$ with endpoints $(0,e)$ and $(1, e)$ together 
with a $1$-cell $(1, e<R)$ with endpoints $(1, e)$ and $(2,e)$.    
Each  length two NCP $\alpha \in \{(123), (124), (134), (234), (12)(34), (14)(23)\}$
contributes a single $1$-cell $(0, e<\alpha)$ with endpoints $(0, e)$ and $(2, e)$.    Finally, 
for each of the $16$ chains of the form $e<R<\alpha$ corresponding to factorisations of 
$\gamma$ by transpositions, 
we have a $2$-cell $(0, e<R<\alpha)$ whose boundary is glued  along the three $1$-cells
$(0, e<R)$, $(0, e<\alpha)$ and $(1, e<R^{-1}\alpha)$.

\begin{figure}[h]
\begin{picture}(100, 100)(130,0)
 \put(125,5){\line(1,0){100}}
 \put(125,5){\line(2,3){50}}
 \put(175,80){\line(2,-3){50}}
  \put(85,5){$(0, e)$}
  \put(160,85){$(2, e)$}
 \put(230,5){$(1, e)$}
\put(150,-7){$(0, e<R)$}
\put(90,45){$(0, e < \alpha)$}
\put(205,45){$(1, e<R^{-1}\alpha)$}
 \end{picture}
\caption{Typical $2$-cell $(0,e<R<\alpha<\gamma)$ of $\FPhat$ for $W = A_3$.}
\label{triangle}
\end{figure}
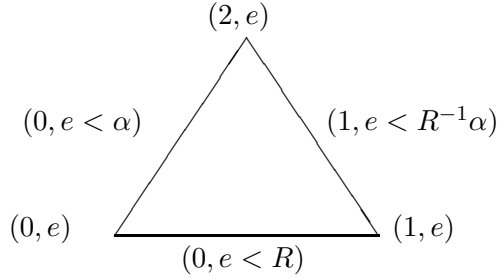
\end{example}

\begin{remark} The cells of $\FPhat$ are simplices, but $\FPhat$ is not a simplicial complex; rather it is a $\Delta$-complex in the sense of \cite{H}.  It can be realized as the nerve of the germ (in the sense of category theory, see \cite{Deh15}) with objects $\{0, 1, \ldots, n-1\}$ and morphisms $i \overset{\alpha}{\longrightarrow} j$
where $\alpha$ is an NCP of length $j-i$.  The composition, when defined, is determined by multiplication of NCPs.
\end{remark} 

\begin{remark}\label{mono}
The monodromy action on $\FPhat$ is obtained by composing the monodromy action on $\X_{\pstar}$ with the retraction defined by the acyclic matching.  It is more convenient to describe the action of $-1 \in \zed$.  
Explicitly a cell of form 
\[(m, e<w_1<w_2<\dots < w_k)\]
for $0< m < n-|w_k|$ is taken to 
\[
 (m-1, e<w_1<w_2<\dots < w_k) 
\]
while
\[(0, e<w_1<w_2<\dots < w_k)\]  for $ |w_k|< n$ is taken to 
\[(|w_1|-1, e<w_1^{-1}w_2<w_1^{-1}w_3<\dots < w_1^{-1}w_k< w_1^{-1}\gamma).\]
The second case is explained as follows.  The action of $-1$ on the $k$-cell 
$(0, e<w_1<w_2<\dots < w_k)$ as a cell in $\X_{\pstar}$ takes it to the 
$k$-cell $(-1, e<w_1<w_2<\dots < w_k)$.  This last $k$-cell is the bottom face 
of the $(k+1)$-cell $(-1, e<w_1<w_2<\dots < w_k< \gamma)$, which in turn has top face  
$(|w_1|-1, e<w_1^{-1}w_2<w_1^{-1}w_3<\dots < w_1^{-1}w_k< w_1^{-1}\gamma)$.
\end{remark}
  
\section{\NCP\ model for the Milnor fiber of the arrangement}
\label{fullfibre}
The homomorphisms $s$ and $\pstar$ from Example~\ref{pure}  and 
 Section~\ref{disc} combine to give the homomorphism
\[\psi=(\pstar,s):B(W) \mapsto \zed \times W: x \mapsto (\pstar(x), s(x)).\]
The space $\X_\psi$ (constructed as in subsection~\ref{quotients})
is then a $K(\pi, 1)$ for $\pi=\ker(\psi)$.



\begin{proposition} The complex $\X_{\psi}$ is homotopy equivalent to the Milnor fiber $F_Q$ of $Q$.
\end{proposition}

\begin{proof} By \cite{Bries71}, the map $f \colon \M \to W \backslash \M \cong \C^n - \Delta$ 
induces an injection $f_* \colon \pi_1(\M) \to B(W)$ whose image is the kernel  of 
$s \colon B(W) \to W$. 
Now
\begin{eqnarray*}
\ker(\psi) &=& \ker(\pstar) \cap \ker(s)\\
&= &\ker(\pstar) \cap \im(f_*) \\
& =& f_*(\ker(\pstar \circ f_*))\\
& =& f_*(\ker((Q^2)_*)) \mbox{ \ \ since } P \circ f = Q^2\\
& =& f_*(\ker(Q_*)) \mbox{ \ \ since } (Q^2)_*=2(Q_*).
\end{eqnarray*}
It follows that  
$X_\psi$ is homotopy equivalent to the cover of $K$ 
corresponding to $f_*(\ker(Q_*))$ and hence to the cover of $W\backslash M$ corresponding to 
$f_*(\ker(Q_*))$.  However, this latter cover is homotopy equivalent to 
$F_Q$ by  Corollary~\ref{fib}(ii).
\end{proof}

\begin{remark}The cover $\X_{\psi}$ is a simplicial complex.  The vertices of the $k$-cell 
$c_k=((P_*(x), s(x)), e< w_1< \cdots < w_k)$
are $((P_*(x), s(x)),e)$ and $((P_*(x)+|w_i|, s(xw_i)),e)$, $1\leq i \leq k$.
This set of $k+1$ vertices uniquely determines $c_k$. 
\end{remark}

\begin{remark}The map $\psi=(\pstar,s)$ is not onto.  For each $x \in B(W)$,  the integer 
$\pstar(x)$ is even 
if and only if $s(x)$ belongs to 
the subgroup  $W^+ < W$ of Section~\ref{discriminant}.    
Thus $\X_{\psi}$ can be identified with the CW-complex whose cells are triples, 
$(m,w, \sigma)$, where $\sigma$ is an initialized  
chain in $L$ and 
the parity of $m \in \zed$ is the same as that of 
$w \in W$.  The top face of  $(m, w, e<w_1<\dots <w_k)$ 
is 
\((m+|w_1|, ww_1, e<w_1^{-1}w_2<\dots <w_1^{-1}w_k)\),
while the other faces are given by
$(m, w, e<w_1<\dots <\widehat{w_i} < \dots <w_k)$ for $1\le i \le k$.
\end{remark}

\begin{definition} 
 We define  the {\em small \NCP\ model}  $\FQhat$  of the Milnor fiber of $Q$ to be the finite subcomplex of 
$\X_{\psi}$ which is the preimage of $\FPhat$ under 
the covering $\X_{\psi}\to \X_{\pstar}$ determined by the subgroup inclusion 
$\ker(\psi) \subseteq \ker(\pstar)$.
\end{definition}
\begin{remark}The complex $\FQhat$  is the union of all cells of the form
\[(0, w, e  \lessdot w_1 \lessdot w_2 \lessdot \dots \lessdot w_{n-1}) \ \mbox{where} \ w \in W^+\]
together with their faces. 
\end{remark}
\begin{proposition} The subcomplex $\FQhat$ is a strong deformation retract of the  cover $\X_{\psi}$.
\end{proposition}
\textbf{Proof:  }    The simplicial complex $\X_{\psi}$ is the cover of $\X_{\pstar}$  
determined by the subgroup inclusion $\mbox{ker}(\pstar,s) <\mbox{ker}(\pstar)$. 
By the homotopy lifting property, the strong deformation retraction of $\X_{\pstar}$ 
onto $\FPhat$ (proposition~\ref{retraction}) is covered by a  strong deformation retraction 
of $\X_{\psi}$ onto $\FQhat$.
\begin{corollary} The finite simplicial complex $\FQhat$ is homotopy equivalent to the Milnor fiber
$F_Q$ and is a $K(\pi, 1)$ for the 
group $\mbox{ker}((\pstar,s))$.
\end{corollary} 
\begin{example} \label{dihedral2}
When $W$ is the dihedral group acting on $\reals^2$ with $t$ 
reflections, $R_1, R_2, \dots , R_t$, 
the subcomplex $\FQhat$ has the following description.    Each vertex $(0,w, e)$ 
has $w \in W^+$, the rotation subgroup of $W$, while each vertex $(1,w, e)$ has $w \in W- W^+$, the set of reflections of $W$.   Furthermore, for each rotation $w \in W^+$ and each reflection $R$ there is another 
reflection $S$ with $w = RS$ giving an edge in $\FQhat$ labelled $(0,w,e<S)$,  starting at $(0,w, e)$ and ending at 
$(1,R, e)$.  
Thus  $\FQhat$ is the complete bipartite graph $K_{t,t}$, which is homotopy equivalent to a bouquet 
of $(t-1)^2$ circles. 
This is consistent with the calculations in  \cite[Theorem 1]{MilOr70}. 
\end{example}
\begin{remark} 
The monodromy action on $\FQhat$ is obtained as in Remark~\ref{mono} by composing the monodromy action on $\X_{\psi}$ with the retraction defined by the acyclic matching.  In this case, the 
action on $\X_{\psi}$ is by shifting the height of cells by multiples of $2$.  
However, since $\FQhat$ is a simplicial complex, it is
sufficient to compute the action on $0$-cells.  We describe the action of $-2 \in \zed$.  

Explicitly, we define the action of $-2$ on $0$-cells by
\[((m,w),e) \mapsto \left\{\begin{array}{cc}
((m-2,w),e)&2\le m \le n-1\\
((n-1, w\gamma),e) & m = 1\\
((n-2,w\gamma) ,e) & m = 0.
\end{array}
\right.
\]

The second case is explained as follows.  The action of $-2$ on the $0$-cell $((1,w),e)$ 
as a cell in $\X_{\psi}$ takes it to the $0$-cell $((-1,w),e)$.  This last $0$-cell is the bottom face 
of the $1$-cell $((-1,w), e< \gamma)$, which in turn has top face  
$((n-1, w\gamma),e)$.  The third case is similar.
\end{remark} 


\section{Structure and homology of \FPhat\ }
Although our model $\FPhat$ of the Milnor fiber of $P$ is not a simplicial complex, 
it does have a combinatorial description as a sequence of mapping cones.   
The domains of the mappings in question are 
order complexes of truncations of  
the non-crossing partition lattice~$L$ and
the lexicographic shellability of $L$ yields a chain complex which 
computes $H_*(\FPhat,\zed)$.
\vskip .2cm
Let $L_{[i,j]}= \{w \in L : i \le |w| \le j\}$ and let
$A_{i} = \{(m,\sigma)\in \FPhat \mid m \ge n-i-1\}$  where 
$i,j$ are integers with $0\le i\le  j \le n$. Note that $A_{i}$ has dimension 
$i$ and that 
\(A_0 \subset A_1 \subset \dots \subset A_{n-1} = \FPhat\).
Define $g_i: |L_{[1,n-i-1]}| \to A_{n-i-2}$ by
 \[w_1<w_2<\dots < w_k \mapsto (i+|w_1|, e<w_1^{-1}w_2< \dots < w_1^{-1}w_k).\]
\begin{proposition} 
\label{cone}
The mapping cone of $g_i$ is cellularly isomorphic to $A_{n-i-1}$.
\end{proposition} 
\begin{proof} 
Note that $|L_{[0,k]}|$ is simplicially isomorphic to the cone on $|L_{[1,k]}|$ 
since $L$ has a unique minimal element $e$.   Under this identification, 
the subcomplex $|L_{[1,k]}|$ of $|L_{[0,k]}|$ corresponds to 
$|L_{[1,k]}|\times \{1\}\subset |L_{[1,k]}|\times [0,1]$.    The map 
$|L_{[0,n-i-1]}| \to A_{n-i-1}$ given by 
\[e<w_1<w_2<\dots < w_k \mapsto (i ,e<w_1<w_2<\dots < w_k).\]
combines with the inclusion of $A_{n-i-2}$ into $A_{n-i-1}$ to give a map
\[\hat{g}_i: (|L_{[1,n-i-1]}|\times [0,1] )\coprod A_{n-i-2} \to A_{n-i-1} . \]
 One can show 
that $ \hat{g}_i$ is an identification map which identifies precisely the pairs 
$(x,0)$ with $ (x',0)$ for each $x,x' \in |L_{[1,n-i-1]}|$ and $(x,1)$ with $g_i(x)$.

\end{proof}
\begin{lemma} 
\label{homology}
For all $q \ge 1$, $H_q(A_p, A_{p-1}) \cong\tilde H_{q-1}(|L_{[1,p]}|)$.
\end{lemma} 
\begin{proof}

The filtration
\(A_0 \subset A_1 \subset \dots \subset A_{n-1} = \FPhat\)
yields
 \begin{eqnarray*}
H_q(A_p, A_{p-1})&\cong&\tilde H_q(A_p/A_{p-1})\\
&\cong&\tilde H_q(\Sigma(|L_{[1,p]}|))  \\
&\cong&\tilde H_{q-1}(|L_{[1,p]}|),
\end{eqnarray*}
where $\Sigma$ denotes suspension and the second isomorphism follows from 
Proposition~\ref{cone}.
\end{proof} 
\begin{definition}   For each $i$, we define the $i$th face map 
$d_i : C_{p-1}( L_{[1,p]}) \to C_{p-2}( L_{[1,p]})$ by 
\[d_i(w_1\lessdot w_2 \lessdot \dots \lessdot w_p) = (-1)^{i-1}(w_1\lessdot w_2 \lessdot \dots 
\lessdot \hat{w_i } \lessdot \dots\lessdot w_p)\]
and the top face map $\Omega : C_{p-1}( L_{[1,p]}) \to C_{p-2}( L_{[1,p-1]})$ by 
\[\Omega(w_1\lessdot w_2 \lessdot \dots \lessdot w_p)=
(w_1^{-1}w_2\lessdot w_1^{-1}w_3 \lessdot \dots \lessdot w_1^{-1}w_p).\]

\end{definition} 
\begin{theorem}   The homology of $\FPhat$ is isomorphic to the homology of the chain complex whose $p$th group is $\tilde H_{p-1}(|L_{[1,p]}|)$ and whose boundary homomorphism 
is given, at the level of chains, by 
$\sum a_\sigma \sigma \mapsto \sum a_\sigma \Omega(\sigma)$.
\end{theorem} 
\begin{proof} 
First note that 
\[H_q(A_p, A_{p-1})\cong \tilde H_{q-1}(|L_{[1,p]}|)\cong \left\{
\begin{array}{cr}
\zed^{n_p} & \mbox{if}\ q = p,\\
0 & \mbox{if}\ q \ne p,
\end{array}
\right.
\]
where the last equality uses 
the lexicographical shellability of the \NCP\ lattices \cite{ABW, Bj}.   
By Theorem~39.4 of \cite{Mun}, 
the homology of $\FPhat$ is isomorphic to the homology of the chain complex with $p$th 
group given by $H_p(A_p, A_{p-1})$ and boundary homomorphism given by the connecting homomorphism of the exact sequence of the triple $(A_{p}, A_{p-1}, A_{p-2} )$. 
\vskip .2cm
It remains to compute the boundary homomorphism from $\tilde H_{p-1}(|L_{[1,p]}|)$ to 
$\tilde H_{p-2}(|L_{[1,p-1]}|)$.  
The isomorphism from $\tilde H_{p-1}(|L_{[1,p]}|)$ to $H_p(A_p, A_{p-1}) $ of 
Lemma~\ref{homology} is induced by 
$b_p: |L_{[1,p]}| \to A_p: \sigma \mapsto (n-p-1,e*\sigma)$, where 
$e*\sigma$ means the simplex 
represented by $e < w_1< \dots <w_l$ 
when $\sigma$  is the simplex represented by $w_1< \dots <w_l$.   
Let  $\sum a_\sigma \sigma$ in $C_{p-1}(|L_{[1,p]}|)$ be a cycle so that
\[
0 = \partial\left(\sum a_\sigma \sigma\right)
= \sum a_\sigma 
\partial(\sigma)
= \sum a_\sigma \sum_i d_i(\sigma).\]
Since each $\sigma$ is maximal, the $(p-1)$-chains $d_i\sigma $ 
and $d_j\sigma $ have different length distributions whenever $i\ne j$.  (The length distribution of 
$w_1< w_2 < \cdots < w_p$ is $(|w_1|,|w_2|,\dots,|w_p|)$.)  It follows that 
\begin{equation}\label{bdry}
\sum a_\sigma d_i(\sigma) = 0  \ \mbox{for each} \ 1\le i \le p.
\end{equation}
Using the fact that the connecting homomorphism is defined on the level of chains by the boundary map in $\FPhat$, we get
 \begin{eqnarray*}
\partial\left( b_p\left(\sum a_\sigma \sigma \right)\right) &=&  
\partial\left(\sum a_\sigma ( n-1-p,e*\sigma)\right)\\
 &=& \sum a_\sigma \partial( n-1-p,e*\sigma)\\
 &=& \sum a_\sigma \left(
 (n-p, e*\Omega(\sigma))+\sum_i( n-1-p,e*d_i\sigma)\right)\\
 &=& \sum a_\sigma 
 (n-p, e*\Omega(\sigma))\ \ \mbox{by \eqref{bdry}}\\
 &=& b_{p-1}\left(\sum a_\sigma \Omega(\sigma)\right).
 \end{eqnarray*}
\end{proof} 


\begin{ack} This project began while the second named 
author was in residence at Dublin City University in Spring, 2012, 
under a Fulbright U.S. Scholars grant.  He expresses gratitude to 
DCU and the Fulbright Commission in Ireland for support and hospitality.
\end{ack}

\bibliographystyle{plain}
\bibliography{noncrossing}
\end{document}